\documentclass[10pt]{amsart}
\usepackage{amsmath}
\usepackage[margin=1.5in]{geometry}
\usepackage{colortbl}
\usepackage{graphicx}
\usepackage{multirow}
\usepackage{bbm}
\usepackage{amssymb}
\usepackage{graphicx}
\usepackage{multirow}
\usepackage{bbm}
\usepackage{bm}
\usepackage{commath}
\usepackage{epsfig}
\usepackage{epstopdf}
\usepackage{color}
\usepackage{subfigure}
\usepackage{epic}
\usepackage{eepic}
\usepackage{graphicx}
\usepackage{array}
\usepackage{mwe}
\usepackage{bm}

\newcommand{\ignore}[1]{}

\title[Consensus Based Optimization algorithm and its Application to Finance ]{A Constrained Consensus Based Optimization algorithm and its Application to Finance}

\author[Bae]{Hyeong-Ohk Bae}
\address[Bae]{Department of Financial engineering, Ajou University, Suwon 16499, Republic of Korea}
\email{hobae@ajou.ac.kr}

\author[Ha]{Seung-Yeal Ha}
\address[Ha]{Department of Mathematical Sciences and Research Institute of Mathematics, Seoul National University, Seoul 08826, Republic of Korea}
\email{syha@snu.ac.kr}
\email[url]{http://www.math.snu.ac.kr/\~{}syha}

\author[Kang]{Myeongju Kang}
\address[Kang]{Department of Mathematical Sciences, Seoul National University, Seoul 08826, Republic of Korea}
\email{bear0117@snu.ac.kr}

\author[Lim]{Hyuncheul Lim}
\address[Lim]{Department of Mathematics, Chonnam National University, Gwangju, Republic of Korea}
\email{limhc@jnu.ac.kr}

\author[Min]{Chanho Min}
\address[Min]{Department of Financial engineering, Ajou University, Suwon 16499, Republic of Korea}
\email{chanhomin@ajou.ac.kr}

\author[Yoo]{Jane Yoo}
\address[Yoo]{Department of Financial engineering, Ajou University, Suwon 16499, Republic of Korea}
\email{janeyoo@ajou.ac.kr}

\newtheorem{theorem}{Theorem}[section]
\newtheorem{lemma}[theorem]{Lemma}
\newtheorem{corollary}[theorem]{Corollary}

\newtheorem{remark}[theorem]{Remark}

\newcommand{\bbr}{\mathbb R}

\newcommand{\bbe}{\mathbb E}

\newcommand{\bbw}{\mathbf{w}}

\newcommand{\bby}{\mbox{${\bf y}$}}

\newcolumntype{x}[1]{%
>{\centering\hspace{0pt}}p{#1}}%

\begin{document}

\date{\today}

{\thanks{H.O. Bae is supported by the Basic Research Program through the National Research Foundation of Korea(NRF) funded by the Ministry of Education and Technology (NRF-2018R1D1A1A09082848). S.-Y. Ha is supported by NRF-2020R1A2C3A01003881. The work of H. Lim is supported by NRF-2019R1I1A3A03059382. The work of C. Min is supported by 
NRF-2021R1G1A1095140. The work of J. Yoo is supported by Ajou University Research Fund.}}
\begin{abstract}
In this paper, we propose a predictor-corrector type Consensus Based Optimization(CBO) algorithm on a convex feasible set.
Our proposed algorithm generalizes the CBO algorithm in \cite{H-J-K} to tackle  a constrained optimization problem for the global minima 
of the non-convex function defined on a convex domain. 
As a practical application of the proposed algorithm, 
we study the portfolio optimization problem in finance.  
In this application, we introduce an objective function to choose the optimal weight 
on each asset in a asset-bundle which yields the maximal expected returns given a certain level of risks. 
Simulation results show that our proposed predictor-corrector type model is successful in finding the optimal value.   
\end{abstract}

\maketitle

\section{Introduction}
\setcounter{equation}{0}
Population-based stochastic optimization algorithms have been extensively used to solve large-scale optimization problems {{arising 
from collective behaviors in nature and human society \cite{A-B-F, C-H-L, G-C-H-Z, P-K-B}, 
and  these types of metaheuristic algorithms are preferred over the gradient-based type algorithms} \cite{Be}. It is because the gradient-based type algorithm requires the computation of gradients of objective functions so that it cannot be used for non-differentiable objective functions which appear in machine learning problems \cite{Bo}. In contrast, population-based algorithms can be used for solving data-driven optimization problems { and problems with non-smooth objective functions} \cite{C-J-L-Z,C-C-T-T,T-P-B-S}.

In this work, we develop a population-based searching algorithm to solve a problem with constraints 
based on the consensus-based searching method namely, CBO algorithm, which is the first-order method
so that intrinsically { it is more straightforward than the second-order particle swarm optimization (PSO) algorithm \cite{G-P}.
The CBO is also simpler and easier to be applied in real problems} than meta-heuristic ones \cite{L-A, Ya}. 
Despite its simplicity, the CBO is robust in solving high dimensional non-convex, non-regular optimization problems 
including artificial intelligence applications \cite{H-J-K}. 
Several variants of the CBO algorithm are further addressed in a series of recent works \cite{C-J-L-Z, C-C-T-T, P-T-T-M, T-P-B-S}.

In \cite{H-J-K2, H-J-K}, the authors dealt with a constrained optimization problem finding a global minima of a twice continuously differentiable function defined 
on a whole Euclidean domain. We suggest a more generalized algorithm for solving such a problem, whose admissible set is closed and convex. 
{ Specifically, we develop a predictor-corrector type CBO algorithm based on the discrete dynamical systems in \cite{H-J-K}. For each iteration, 
we apply their discrete method to find the optimal of an objective function (\textit{prediction step}) 
and then project the predicted value onto an admissible set (\textit{correction step}). }
After introducing this predictor-corrector type algorithm, 
we provide {sufficient conditions on system parameters and initial data}  for the convergence to a global optima 
of a continuously differentiable function, which is defined on a closed and convex feasible set. 

We apply our predictor-corrector type algorithm to the optimal portfolio selection problem in finance. 
{To develop suitable optimization algorithms for tackling large instances of the constrained optimal portfolio selection problem, 
several researchers and practitioners have suggested meta-heuristic methods \cite{Ya, Y-D} including machine learning techniques,
(hierarchical) clustering \cite{N}, genetic algorithm (GA) \cite{Ho}, and PSO algorithms \cite{E-K} to find high-quality solutions
in a reasonable amount of time. }
Motivated by the aforementioned results of consensus-based optimizing algorithms, we provide a sufficient framework of 
our predictor-corrector type algorithm 
to solve the constrained inter-temporal portfolio selection problem, in which an agent maximizes portfolio's return given a certain level of risks. 
To the best of our knowledge, this is the first study on the CBO algorithm applied to a financial problem. 

\vspace{0.2cm}
The rest of this paper is organized as follows. 
{In Section \ref{sec:2}, we propose a predictor-corrector CBO algorithm on the convex feasible set.}  
In Section \ref{sec:3}, we discuss mathematical properties of the proposed predictor-corrector type algorithm. 
Finally, Section \ref{sec:4} is devoted to financial preliminaries, and numerical simulations of the proposed CBO algorithm 
in optimal portfolio selection problem.

\vspace{.5cm}
\noindent Notation: {Throughout the paper, $N$ denotes the number of particles(agents) and we set $[N] := \{1, 2, \cdots, N\}$. We also identify $i \in [N]$ with the agent $ i $. }

\section{CBO Algorithm on a convex feasible set} \label{sec:2}
\setcounter{equation}{0}

Let $\bbw_n^k = (w_n^{k,1}, \cdots, w_n^{k,d}) \in \bbr^d$ be the value of the $k$-th sample path at $ n $-th step. We look for a global optimizer $\bbw^* \in \bbr^d$ for a given continuous objective function $L$ defined on a nonempty closed convex feasible set $\mathcal S$:
\begin{align*}
{  L^* := L(\bf w^*) = \underset{\bbw\in\mathcal S}{\min} \ L(\bbw).}
\end{align*}

Given $L$, the CBO model is proposed to find a global minimizer $\bbw^{*}$ in \cite{H-J-K2}. 
Note that the approximate solution $\bbw_n^k$ to the algorithm in \cite{H-J-K2} is generally not in the feasible set $ \mathcal{S}$. To enforce the approximate solution in $\mathcal S$, { we propose a predictor-corrector type CBO algorithm by first forwarding Euler scheme(\textit{prediction step}), and then projecting the predicted value to $\mathcal{S}$ in each iteration(\textit{correction step})}:
	\begin{equation} \label{dCBO2}
		\begin{cases}
			\displaystyle \hat{\bbw}^i_{n+1} = \bbw_n^i -\lambda h({\bbw}^i_n - {\bar {\bbw}}_n)
			+\sigma\sqrt h ({\bbw}_n^i- {\bar {\bbw}}_n)\odot \eta_n , \quad n \geq 0,~~ i \in [N], \vspace{.3cm}\\
			\displaystyle {\bbw}_{n+1}^i = {\mathbb P}_\mathcal S [\hat{\bbw}_{n+1}^i], \vspace{.3cm}\\
			\displaystyle {\bar {\bbw}}_n =(\bar{w}_n^{1}, \cdots, \bar{w}_n^{d}) 
			:= \frac{\sum_{l=1}^{N} {\bbw}^l_n e^{-\beta L({\bbw}^l_n)}}{\sum_{l=1}^{N} e^{-\beta L({\bbw}^l_n)}},  \hspace{0.5cm} \bbw_n^i \big|_{n =0} = \bbw^{i}_0,
		\end{cases}
	\end{equation}
where $\lambda, \sigma$ and $\beta$ denote the drift rate, noise intensity and positive hyperparameter, respectively. In statistical physics, $\beta$ often corresponds to the reciprocal of temperature(coldness). Here, $\eta_n:=[\eta_n^1,\cdots,\eta_n^d]^\top$, and $A\odot B$ represents the Hadamard product, 
i.e. $A\odot B=(a_{ij}b_{ij})$ for matrices $A=(a_{ij})$ and $B=(b_{ij})$ of the same size. 
We assume that the one-dimensional random variable $\eta_n^l$ is i.i.d. with its zero mean and covariance such that  
\begin{align*}
\bbe[\eta_n^l] &= 0 \quad \mbox{for $l = 1, \cdots, d$},  \\
 \bbe[\eta_n^{l_1} \eta_k^{l_2}] &= \delta_{l_1 l_2}\delta_{nk}, \quad \mbox{for} \quad 1 \leq  l_1, l_2 \leq d, \quad 0\leq n,k,
\end{align*} 
	and ${\mathbb P}_{\mathcal{S}} $ is a projection onto $ \mathcal{S} $:
\begin{align*}
{\mathbb P}_{\mathcal S}[\hat {\bbw}] := \arg\min_{\bbw\in \mathcal{S}} \|\bbw-\hat{\bbw}\|^2.
\end{align*}
Here, $\| \cdot \|$ denotes the standard $\ell_2$-norm in $\bbr^d$. Note that $ \bbw_k$,  $k\le n$, is independent of $ \eta_n^{l} $.



\section{Discrete CBO Model on Convex Domain} \label{sec:3}
\setcounter{equation}{0}
In this section, we discuss our main results. In relation with the correction step in \eqref{dCBO2}, we first recall the following elementary lemma. 
\begin{lemma}\label{proj}
	{Let $\mathcal{S}$ be a nonempty closed and convex domain. Then $ \mathbb{P}_{\mathcal{S}}$ satisfies a contraction property: }
\begin{align*}
\| {\mathbb P}_{\mathcal{S}}(\hat \bbw)- {\mathbb P}_{\mathcal{S}}(\hat \bby)\| \leq \|\hat \bbw-\hat \bby\|, \qquad\forall~\hat{\bbw},\hat{\bby}.
\end{align*}
\end{lemma}
\begin{proof} Since the proof is rather standard and elementary, we omit its proof here. 
\end{proof}
Now, we state our first main result on the emergence of stochastic consensus.
\begin{theorem}
\emph{(Emergence of global consensus)} \label{T3.2}
{ Suppose the system parameters satisfy}
		\begin{align}\label{L3.3b}
		\sigma > 0, \quad   2\lambda > \sigma^2 \quad \mbox{and} \quad 0 <  h < \frac{2\lambda - \sigma^2}{\lambda^{2}}.
		\end{align}	
Then, for a solution $\{\bbw_n^i \}$ to  \eqref{dCBO2}, one has 
\[ \mathbb E \big[ \|\bbw_n^i-\bbw_n^j \|^2 \big] \leq e^{-nhm} \bbe \big[ \|\bbw^i_0 - \bbw^j_0 \|^2 \big], \quad n \in {\mathbb N}, \]
where $ m:= (2\lambda - \lambda^2 h -\sigma^2)$.
\end{theorem}
\begin{proof}
	From \eqref{dCBO2},  
\begin{align*}
	\begin{aligned} 
		\hat{w}^{i,l}_{n} - \hat{w}^{j,l}_{n}  &= w^{i,l}_n - w^{j,l}_n -\lambda h ( w^{i,l}_n - w^{j,l}_n) - \sigma \sqrt{h} (w^{i,l}_n - w^{j,l}_n) \eta_n^l \\
		&= \Big(1 -\lambda h - \sigma \sqrt{h} \eta_n^l  \Big) (w^{i,l}_n - w^{j,l}_n).
	\end{aligned}
\end{align*}
\noindent This implies
\begin{align*}
(\hat{w}^{i,l}_{n} - \hat{w}^{j,l}_{n}  )^2
	= ( 1 -\lambda h - \sigma \sqrt{h} \eta_n^l )^2 ( w^{i,l}_n - w^{j,l}_n )^2.
	\end{align*}
Now, we take expectation from both sides to get
\begin{equation}\label{eq3.2}
\begin{aligned}
	\mathbb{E}[(\hat{w}^{i,l}_{n} - \hat{w}^{j,l}_{n}  )^2]
	&= \mathbb{E}[( 1 -\lambda h - \sigma \sqrt{h} \eta_n^l )^2 ( w^{i,l}_n - w^{j,l}_n )^2]\\
	&=\mathbb{E}[( 1 -\lambda h - \sigma \sqrt{h} \eta_n^l )^2]\mathbb{E}[ ( w^{i,l}_n - w^{j,l}_n )^2]\\
	&=(1 -2\lambda h +\lambda^2 h^2+\sigma^2h)\mathbb{E}[ ( w^{i,l}_n - w^{j,l}_n )^2],
\end{aligned}
\end{equation}
where we use the independence of $ \eta_n^{l} $ and $ (w_n^{i,l}-w_{n}^{j,l}) $. We sum up over $ l $ in \eqref{eq3.2}, we have 
\begin{equation*}
	\mathbb{E} \big[ \|\hat{\bbw}_n^{i}-\hat{\bbw}_n^{j}\|^2 \big] =	(1-h\underbrace{(2\lambda-\lambda^2h-\sigma^{2})}_{=: m(\lambda,h, \sigma)})\mathbb{E} \big[ \|{\bbw}_n^{i}-{\bbw}_n^{j}\|^2 \big].
\end{equation*}
{Next, we use Lemma \eqref{proj} and inequality $ 1+x\leq e^x  $} to obtain 
\begin{equation*}
	\mathbb E \big[ \|\bbw_{n+1}^i-\bbw_{n+1}^j \|^2 \big] \leq e^{-h(2\lambda - \lambda^2 h -\sigma^2)} \bbe \big[ \|\bbw^i_n - \bbw^j_n \|^2 \big].
\end{equation*}
This yields
\begin{equation*} 
\mathbb E \big[ \|\bbw_n^i-\bbw_n^j \|^2 \big] \leq e^{-nh(2\lambda - \lambda^2 h -\sigma^2)} \bbe \big[ \|\bbw^i_0 - \bbw^j_0 \|^2 \big], \quad n \in {\mathbb N}.
\end{equation*}

\end{proof}

\begin{corollary} \label{C3.3}
{Under the same setting as in Theorem 3.2 together with the following extra assumptions on the initial distribution of $ \bbw_0^i $:
\begin{equation} \label{New-1}
	\quad \bbw_0^i: i.i.d., \quad \bbw_0^i \sim \bbw_{in},
\end{equation} one has 
\begin{align*}
& \mathbb E \big[ \|\bbw_n^i-\bar \bbw_n \|^2 \big] \leq 2\left(\frac{d-1}{d}\right)^2e^{-nhm} \left(\mathbb E\Big[ \|\bbw_{in}\|^2\Big]-\mathbb E\Big[\bbw_{in}\Big]\cdot \mathbb E\Big[\bbw_{in}\Big]\right), \\
& \mathbb E\big[ \|\bbw_n^i-\bar \bbw_n \|\big] 
\leq \Big(\frac{d-1}{d}\Big)e^{-\frac{nhm}{2}}\sqrt{ 2\left(\mathbb E\Big[ \|\bbw_{in}\|^2\Big]-\mathbb E\Big[\bbw_{in}\Big]\cdot \mathbb E\Big[\bbw_{in}\Big]\right)}.
\end{align*}
}
\end{corollary}
\begin{proof}
{	(i) By Jensen's inequality and Theorem \ref{T3.2}, we have 
	\begin{align*}
	\begin{aligned}
		\mathbb E \big[\|\bbw_n^i-\bar \bbw_n \|^2\big] &=\mathbb E\bigg[ \bigg\|\frac{1}{d}\sum_{j=1,j\neq i}^{d}(\bbw_n^i- \bbw_n^{j}) \bigg\|^2 \bigg]\leq \left(\frac{{d-1}}{d^{2}}\right)\mathbb E \bigg[ \sum_{j=1,j\neq i}^{d} \| \bbw_n^i- \bbw_n^{j} \|^2 \bigg]\\
		&\leq\left(\frac{d-1}{d^{2}}\right)\sum_{j=1,j\neq i}^{d} e^{-nhm}\mathbb E[ \|(\bbw_0^i- \bbw_0^{j}) \|^2]\\ 
		&=\left(\frac{d-1}{d^{2}}\right)\sum_{j=1,j\neq i}^{d} e^{-nhm}\mathbb E \Big[ \|\bbw_0^i\|^2+ \|\bbw_0^{j} \|^2 -2\bbw_0^i\cdot \bbw_0^j\Big]\\
		&=2\left(\frac{d-1}{d}\right)^2 e^{-nhm} \left(\mathbb E\Big[ \|\bbw_{in}\|^2\Big]-\mathbb E\Big[\bbw_{in}\Big]\cdot \mathbb E\Big[\bbw_{in}\Big]\right),
		\end{aligned}
	\end{align*}
	where in the last equation we used the i.i.d. conditions of initial distribution $ \bbw_0^{i} $. 
	
	\vspace{0.2cm}
	
\noindent (ii)~Again we use Jensen's inequality to get the second estimate:
\begin{equation*}
	\left(\mathbb E \|\bbw_n^i-\bar \bbw_n \|\right)^2 \leq \mathbb E \|\bbw_n^i-\bar \bbw_n \|^2\leq 
	2\left(\frac{d-1}{d}\right)^2 e^{-nhm} \left(\mathbb E\Big[ \|\bbw_{in}\|^2\Big]-\mathbb E\Big[\bbw_{in}\Big]\cdot \mathbb E\Big[\bbw_{in}\Big]\right).
\end{equation*}
}
\end{proof}
For next lemma, we set
\begin{align}
\begin{aligned} \label{key}
& \mathrm{Var}(\bbw_{in}):=\mathbb E\Big[ \|\bbw_{in}\|^2\Big]- \Big \|\mathbb E\Big[\bbw_{in}\Big] \Big\|^2, \\
& A_n := \sum_{p=0}^{n} \|\bbw_p^{i}-\bar{\bbw}_p \| \quad \mbox{and} \quad  B_n := \sum_{p=0}^{n} \|({\bbw}_p^i- {\bar {\bbw}}_p)\odot \eta_p \|.
\end{aligned}
\end{align}

\begin{lemma}\label{L3.4}
 The expectation values of sequences $ A_n $ and $ B_{n} $ are uniformly bounded for every $ n $: there exists a positive constant $C$ independent of $n$ such that 
	\[
	 \sup_{1 \leq n < \infty} \mathbb{E}[A_n] \leq C  \quad \mbox{and} \quad  \sup_{1 \leq n < \infty} \mathbb{E}[B_n]\leq C.
	\]
\end{lemma}
\begin{proof}
First, we calculate each term in $ B_n $ as
\begin{equation*}
	\|({\bbw}_p^i- {\bar {\bbw}}_p)\odot \eta_p\|^{2}=\sum_{l=1}^{d}\left((w_p^{i,l}-\bar w_p^{l})\eta_p^{l}\right)^2.
\end{equation*}
We take expectation on both sides to find
\begin{align} \label{C-2}
\begin{aligned}
\mathbb E \big[ \|({\bbw}_p^i- {\bar {\bbw}}_p)\odot \eta_p\|^{2} \big]
&=\sum_{l=1}^{d}\mathbb E \left[\left((w_p^{i,l}-\bar w_p^{l})\eta_p^{l}\right)^2\right] =\sum_{l=1}^{d}\mathbb E \left[\left(w_p^{i,l}-\bar w_p^{l}\right)^2\right]\mathbb E\left[(\eta_p^{l})^2\right] \\
&=\sum_{l=1}^{d}\mathbb E \left[\left(w_p^{i,l}-\bar w_p^{l}\right)^2\right] = \mathbb E \left[\|\bbw_p^{i}-\bar \bbw_p\|^2\right].
\end{aligned}
\end{align}
By Jensen's inequality, \eqref{C-2} and Corollary \ref{C3.3}, we have
\begin{align*}
	\mathbb E 	\|({\bbw}_p^i- {\bar {\bbw}}_p)\odot \eta_p\|\leq \sqrt{\mathbb E \big[ \|({\bbw}_p^i- {\bar {\bbw}}_p)\odot \eta_p\|^{2} \big]} \leq \Big(\frac{d-1}{d}\Big) e^{-\frac{phm}{2}}{\sqrt{ 2\mathrm{Var}(\bbw_{in})}}.
\end{align*} 
Now, we take a summation over $p$ to obtain
\begin{align*}
\begin{aligned}
\bbe[B_n] &\leq \Big(\frac{d-1}{d}\Big){\sqrt{ 2\mathrm{Var}(\bbw_{in})}}\sum_{p=0}^n e^{-\frac{phm}{2}} \leq \Big(\frac{d-1}{d}\Big){\sqrt{ 2\mathrm{Var}(\bbw_{in})}}\sum_{p=0}^\infty e^{-\frac{phm}{2}} \\
& \leq \Big(\frac{d-1}{d}\Big){\sqrt{ 2\mathrm{Var}(\bbw_{in})}}\frac{1}{1-e^{-\frac{hm}{2}}}.
\end{aligned}
\end{align*}
This implies the uniform boundedness of $\bbe[B_n]$. Again, we use Corollary \ref{C3.3} again to see the uniform boundedness of $ \mathbb{E}[A_n] $:
\begin{align*}
\bbe[A_n] &= \mathbb E \bigg[ \sum_{p=0}^{n}\|\bbw_{p}^{i}-\bar\bbw_{p}\| \bigg] \leq \Big(\frac{d-1}{d}\Big){\sqrt{ 2\mathrm{Var}(\bbw_{in})}} \sum_{p=0}^{n} e^{-\frac{phm}{2}} \\
	&\leq  \Big(\frac{d-1}{d}\Big){\sqrt{ 2\mathrm{Var}(\bbw_{in})}}\frac{1}{1-e^{-\frac{hm}{2}}}.
\end{align*}
\end{proof}

\begin{theorem} 
\emph{(Emergence of a global consensus state)} \label{T3.5}
Suppose the system parameters satisfy \eqref{L3.3b}, and let $\{\bbw_n^i \}$ be a solution to  \eqref{dCBO2}. Then, there exists a random vector $~\bbw_\infty~ \mbox{such that}~$
\begin{equation*} 
	\begin{aligned}
\lim\limits_{n\to\infty}  \bbw_n^i=\bbw_\infty \quad \mbox{a.s.}, \quad i \in [N].
	\end{aligned}
\end{equation*}
\end{theorem}
\begin{proof}
We prove this theorem by showing that ${\bbw}_n^i $ is a Cauchy sequence for a.s. $ \omega \in \Omega $ in the compact subset of $ \mathcal{S} $. Note that 
	\begin{equation}\label{3.15}
		\begin{aligned}
			\|\bbw_m^{i}-\bbw_n^{i} &\|\leq \sum_{p=n}^{m-1} \|\bbw_{p+1}^{i}-\bbw_p^{i} \|\leq \sum_{p=n}^{m-1} \|\hat{\bbw}_{p+1}^{i}-\bbw_p^{i} \| \\
			&\leq \sum_{p=n}^{m-1}\lambda h \|\bbw_p^{i}-\bar{\bbw}_p \|+\sum_{p=n}^{m-1}\sigma\sqrt h \|({\bbw}_p^i- {\bar {\bbw}}_p)\odot \eta_p \| \\
			&=\lambda h(A_{m-1}-A_{n-1})+\sigma\sqrt{h}(B_{m-1}-B_{n-1}).
		\end{aligned}
	\end{equation}

By Lemma \ref{L3.4}, it is clear that $ A_n$ and $ B_n$ are uniformly bounded submartingale.
 Hence, by Doob's martingale convergence theorem, 
 \[ A_n\to A_{\infty}  \quad \mbox{and} \quad  B_n\to B_{\infty} \quad \mbox{a.s.}~~n \to \infty, \]
 for some $A_{\infty}, B_{\infty}$. Thus,
  \[  (A_n)  \quad \mbox{and} \quad  (B_n)~\mbox{are Cauchy for a.s. $\omega$}. \]
   Therefore, for a given $ \varepsilon $, we can choose a sufficiently large $ N $ such that
\begin{equation}\label{NE-1}
m, n > N \quad \Longrightarrow \quad \max\big\{ |A_m-A_n|, ~|B_m-B_n| \big\} \leq \dfrac{\varepsilon}{2} \quad \mbox{a.s.} ~\omega. 
\end{equation}
Then, \eqref{3.15} and \eqref{NE-1} imply that $\{ \bbw_n^{i} \}$ is Cauchy. Therefore, there exists ${\bbw}_\infty ^i $ such that $ \lim_{n\to\infty} \bbw_n^{i}=\bbw_{\infty }^{i} $ a.s.. Finally, we use the global consensus result to show that $ w_n^i $ converges to the $ \bbw_\infty $ common consensus state a.s.. 
\end{proof}

\begin{theorem}[Error estimate] \label{T3.6}
Assume that the following conditions hold. 
\begin{enumerate}
\item
Suppose that the objective function $L = L(\bbw)$ is bounded { and Lipschitz continuous with Lipschitz constant  $ C_L$ and lower bound $L^*$}:
\[ |L(\bbw_1) - L(\bbw_2)| \leq C_L |\bbw_1 - \bbw_2|, \quad \forall~\bbw_1, \bbw_2 \in {\mathbb R}, \quad \mbox{and} \quad \inf_{\bbw \in {\mathbb R}}  L(\bbw) = L^*. \]
\item
Suppose system parameters, initial data satisfy
	\begin{align*}
	& 2\lambda>\sigma^2, \quad \bbw_0^i: i.i.d., \quad \bbw_0^i \sim \bbw_{in}, \\
	& (1-\varepsilon)\mathbb E \Big[ e^{-\beta L(\bbw_{in})} \Big ] \geq \frac{\beta C_Le^{-\beta L^*}(d-1)(\lambda h + \sigma \sqrt{h})}{d(1-e^{-\frac{hm}{2}})}\sqrt{2\mathrm{Var}(\bbw_{in})}
	\end{align*}
	for some $0<\varepsilon<1$ and a random variable $\bbw_{in}$ appearing in \eqref{New-1}. 
	\end{enumerate}
	and let $\{\bbw_n^i \}$ be a solution process to \eqref{dCBO2}. Then, there exists a function $E(\beta)$  such that 
	\[
	\lim\limits_{\beta\to\infty}E(\beta)=0 \quad \mbox{and} \quad \operatorname{essinf}_{\omega\in\Omega} L({\bbw}_\infty(\omega))\leq \operatorname{essinf}_{\omega\in\Omega} L({\bbw}_{in}(\omega))+E(\beta).
	\]
	{ In particular, if the global minimizer $\bbw^*$ of $L$ is contained in the support of $\operatorname{law}({\bbw}_{in})$, then}
	\[
	\operatorname{essinf}_{\omega\in\Omega} L({\bbw}_\infty(\omega))\leq L^*+E(\beta).
	\]
\end{theorem}
\begin{proof}
 {We slightly improve arguments in the proof of Theorem 3.2 in \cite{H-J-K} by replacing $\mathcal C^2$-regularity assumption of the objective function $L$} by Lipschitz continuity. More precisely, we proceed our proof by avoiding estimates for $\nabla^2 L$. Since overall proof can be found in \cite{H-J-K}, we focus on the key ingredient part: since $e^x-1\geq x$ for all $x\in\bbr$, we have
	\begin{align}
		\begin{aligned} \label{D-13a}
			&\frac{1}{N}\sum_{i=1}^N  e^{-\beta L(\bbw_{n+1}^i)} -\frac{1}{N}\sum_{i=1}^N  e^{-\beta L(\bbw_n^i)} =\frac{1}{N}\sum_{i=1}^N  e^{-\beta L(\bbw_{n}^i)}(e^{-\beta(L(\bbw^{i}_{n+1})-L(\bbw_n^i))}-1)\\
			& \hspace{.5cm} \geq \frac{1}{N}\sum_{i=1}^N  e^{-\beta L(\bbw_{n}^i)}(-\beta)(L(\bbw_{n+1}^{i})-L(\bbw_{n}^{i})) \\
			&\hspace{.5cm} =-\frac{\beta}{N}\sum_{i=1}^Ne^{-\beta L(\bbw_{n}^i)} \Big( \nabla L (c\bbw_{n+1}^{i}+(1-c)\bbw_{n}^{i}) \Big) \cdot (\bbw_{n+1}^{i}-\bbw_{n}^{i}) \\
		& \hspace{.5cm} \geq-\frac{\beta}{N}\sum_{i=1}^Ne^{-\beta L(\bbw_{n}^i)}C_{L} \|\bbw_{n+1}^{i}-\bbw_{n}^{i} \| \\
		&\hspace{.5cm} \geq-\frac{\beta C_Le^{-\beta L^*}}{N}\sum_{i=1}^N \|\bbw_{n+1}^{i}-\bbw_{n}^{i} \|
			\geq -\frac{\beta C_Le^{-\beta L^*}}{N}\sum_{i=1}^N \| \hat{\bbw}_{n+1}^{i}-\bbw_{n}^{i} \|,
		\end{aligned}
	\end{align}
{ where the last inequality follows from Lemma \ref{proj} since $ \bbw_{n+1}^{i} $ is the projection of $ \hat{\bbw}_{n+1}^{i} $ and $ \bbw_n^i $ is the projection of itself.}
On the other hand, we use \eqref{C-2} and Corollary \ref{C3.3} to see
\begin{equation} \label{C-3}
	\begin{aligned}
	\bbe \|&\hat{\bbw}_{n+1}^{i}-\bbw_{n}^{i} \|
		{=\bbe \| -\lambda h({\bbw}^i_n - {\bar {\bbw}}_n)
		+\sigma\sqrt h ({\bbw}_n^i- {\bar {\bbw}}_n)\odot \eta_n\|} \\
	&	{\leq\lambda h\bbe \| \bbw_n^{i}-\bar{\bbw}_n \| +\sigma\sqrt h\bbe \|({\bbw}_n^i- {\bar {\bbw}}_n)\odot \eta_n \| } \\
	&\leq  \lambda h \bbe \| \bbw_n^{i}-\bar{\bbw}_n \|+\sigma\sqrt h\sqrt{\bbe\Big[ \|{\bbw}_n^i- {\bar {\bbw}}_n \|^2\Big]}\\
	&=(\lambda h + \sigma \sqrt{h})\Big(\frac{d-1}{d}\Big)e^{-\frac{nhm}{2}}{\sqrt{ 2\mathrm{Var}(\bbw_{in})}}.
	\end{aligned}
\end{equation}
We sum up \eqref{D-13a} over $n$, and apply expectation on the resulting relation to get 
	\begin{align*}
	\begin{aligned} 
		&\bbe \bigg[\frac{1}{N}\sum_{i=1}^N  e^{-\beta L(\bbw_n^i)} \bigg]
		\geq \bbe \bigg[\frac{1}{N}\sum_{i=1}^N  e^{-\beta L(\bbw_0^i)} \bigg]  -\frac{\beta C_Le^{-\beta L^*}}{N}\sum_{p=0}^{n-1}\sum_{i=1}^N\bbe \|\hat{\bbw}_{p+1}^{i}-\bbw_{p}^{i} \| \\
		&\hspace{1cm} \geq \bbe [  e^{-\beta L(\bbw_{in})}] -\frac{\beta C_Le^{-\beta L^*}(d-1)}{Nd}(\lambda h + \sigma \sqrt{h})\sum_{i=1}^N\frac{1-e^{-\frac{nhm}{2}}}{1-e^{-\frac{hm}{2}}}{\sqrt{ 2\mathrm{Var}(\bbw_{in})}}\\
		&\hspace{1cm} \geq \bbe [  e^{-\beta L(\bbw_{in})}] -\beta C_Le^{-\beta L^*}(\lambda h + \sigma \sqrt{h})\left(\frac{d-1}{d}\right)\frac{1-e^{-\frac{nhm}{2}}}{1-e^{-\frac{hm}{2}}}\sqrt{2\mathrm{Var}(\bbw_{in})},
	\end{aligned}
\end{align*}
where we used \eqref{C-3} in the second inequality. \newline

{ As $ n \to \infty  $, we have 
\begin{equation*} 
	\begin{aligned}
			\bbe \big[ e^{-\beta L(\bbw_\infty )} \big]
			&\geq \bbe [  e^{-\beta L(\bbw_{in})}] -\frac{\beta C_Le^{-\beta L^*}(d-1)(\lambda h + \sigma \sqrt{h})}{d(1-e^{-\frac{hm}{2}})}\sqrt{2\mathrm{Var}(\bbw_{in})} \geq \varepsilon\bbe e^{-\beta L(\bbw_{in})}.
	\end{aligned}
\end{equation*}
We take the logarithm on both sides of the above inequality to get 

\begin{align*}
	-\frac{1}{\beta}\log \mathbb E \big[ e^{-\beta L(\bbw_{\infty})} \big] \leq-\frac{1}{\beta}\log \mathbb E \big[ e^{-\beta L(\bbw_{in})} \big] -\frac{1}{\beta}\log\varepsilon.
\end{align*}
We combine the above relation and 
\begin{align*}
	\operatorname{essinf}_{\omega\in\Omega} L(\bbw_\infty)&=-\frac{1}{\beta}\log   e^{-\beta \operatorname{essinf}_{\omega\in\Omega} L(\bbw_\infty)} \leq-\frac{1}{\beta}\log \mathbb E \big[ e^{-\beta L(\bbw_{\infty})} \big] 
\end{align*}
 to derive
 \begin{equation*}
 		\operatorname{essinf}_{\omega\in\Omega} L(\bbw_\infty)  \leq-\frac{1}{\beta}\log \mathbb E \big[ e^{-\beta L(\bbw_{in})} \big] -\frac{1}{\beta}\log\varepsilon.
 \end{equation*}

\noindent Now we use Laplace's principle:
\begin{equation*}\label{key}
\lim_{\beta \to \infty}	-\frac{1}{\beta}\log \mathbb E \big[ e^{-\beta L(\bbw_{in})} \big] = L^*.
\end{equation*}
So if we define 
\[ E(\beta):= -\frac{1}{\beta}\log \mathbb E \big[ e^{-\beta L(\bbw_{in})} \big] - L^* -\frac{1}{\beta}\log \mathbb \varepsilon, \]
we have 
\begin{equation*}
	\operatorname{essinf}_{\omega\in\Omega} L(\bbw_\infty)  \leq L^* + E(\beta),
\end{equation*}
where $ \lim _{\beta \to \infty} E(\beta)=0 $ which yields our main result. }
\end{proof}

\hspace{0.2cm}

{
\begin{remark} 
1. Below, we mention advantages and disadvantages of the results from this paper
 compared to the result of \cite{H-J-K} as follows. 
	\begin{itemize}	
	\item The regularity condition of the objective function $ L = L(\bbw) $ is relaxed drastically from the bounded $\mathcal{C}^2 $-regularity to  Lipschitz regularity so that the result of Theorem \ref{T3.6} can be applied to a more broad class of objective functions.
	\vspace{0.1cm}
	\item The convergence results holds for any closed convex set $ \mathcal{S} $, whereas the result of \cite{H-J-K} deals with unconstrained problem. 
	\vspace{0.1cm}
	\item Due to relaxed regularity condition on $L$, we can only show the zero convergence of the error function $ E(\beta) $ without any decay rates. Note that bounded $C^2$-regularity assumption of $L$ yields the decay estimate $E(\beta) = {\mathcal O}(\frac{1}{\beta})$ as $\beta \to \infty$. 
	\end{itemize}

\vspace{0.2cm}

\noindent 2. We also further extended the result of [10] to more general setting using $ \mathrm{Var}(\bbw_{in}) $ instead of $ \max_{i,j} \bbe\|\bbw_0^i-\bbw_0^j\|^2 $ in \cite{H-J-K}. This observation allows more general initial distribution such as Gaussian distribution which has an infinite support. 
\end{remark}
}
\section{Financial Application with Numerical Simulation} \label{sec:4}
\setcounter{equation}{0}
In this section, we apply the suggested predictor-corrector type CBO algorithm to the optimal portfolio selection problem.
\subsection{The Markowitz problem}
The fundamental theory of portfolio optimization is based on Markowitz's seminal work \cite{M}, 
which selects the weight of investors' investment in risky assets based on mean-variance analysis. In what follows, we discuss Markowitz's modern portfolio theory (MPT) \cite{M}
for choosing the optimal combination of a risk-free and multiple risky assets to yield maximal portfolio returns given portfolio's risk. 

Consider a financial market in which one risk-free and $N-1$ risky assets are traded continuously in a finite horizon $[0,T]$. 
The risk-free bond is evaluated by its price $P_t$ for $t \ge 0$, and it evolves according to the following equation: 
\begin{align*} 
dP_t = r_f P_t dt, \quad \mbox{for} \quad t \in (0, T],
\end{align*} 
where $r_f$ is the risk-free interest rate, whereas the risky security $i$ follows the geometric Brownian motion such as 
\begin{align*}
dS_{t}^{i} = \mu^i S_{t}^i dt + \sigma^i S_{t}^i dB^i_t, \quad \mbox{for} \quad t \in (0, T], 
\end{align*} 
where $\mu^i$,$\sigma^i$ and $B^i_t$ are the drift or mean rate of return,  volatility, and a standard Brownian motion of the risky security $i$, respectively.
{Here, $ S_t^i $ denotes the price of a risky security $ i $ at $ t $.}

Suppose that the initial capital $\mathcal{W}_0$ is given at the beginning of the whole investment period. The capital would be invested in a risk-free bond and $N-1$ risky securities at time $t = 0$. 
Let $w_{t}^i$ be the relative amount invested in a security $i$, where $\sum_{i=1}^{n} w_{t}^i$ = 1. Here, $w_t^i$ is considered as a \textit{weight} of $\mathcal{W}_0$ invested in an asset $i$ for $i = 1, \ldots, N$. 

For example, if an investor has one risk-free and one risky assets in his portfolio, 
the wealth process $\mathcal{W}_t$ corresponds to the universal portfolio $\bbw = (w^{1}, w^{2})$ satisfying   
\begin{align} \label{wealthprocess}
d\mathcal{W}_t = \left[w^1_t r_f +w^2_t\left(\mu - r_f\right)\right] \mathcal{W}_t dt + \mathcal{W}_t w^2_t \sigma dB_t, \quad \mbox{for} \quad t \in (0, T],
\end{align} 
{ where $ \mu $ and $ \sigma $ denote a risky asset's expected return and volatility (risk), which are fixed over the entire investment time horizon from $ 0 $ to $ T $. Thus, an investor's expected value of wealth evolves according to the portfolio's expected return (the first term on the right hand side of 4.1) at every t. The portfolio's wealth value also depends on the risky asset's volatility (the second term on the right hand side of equation \ref{wealthprocess}) along with the Brownian motion.}

Let $\bbw^0 = (w_{0}^1, \ldots, w_{0}^N)$ be the initial weights invested in the $N$ assets. 
Suppose that we rebalance our portfolio at times $t = 1, \ldots, T-1$ 
without injecting additional cash into it. At each period $t$, 
we make a portfolio rebalancing decision about 
redistributing the current wealth $\mathcal{W}_t$ among $N$ assets. 
Optimization is performed over all implementable and admissible policies of the optimal weights $w_{t}^i$ 
given realized asset returns $\mathbf{r}_t = (r_{t}^1, \ldots, r_{t}^{N})^{\top}$ for each period $t = 1, \ldots, T$,  
forming a vector-valued random process $\mathbf{r} = (\mathbf{r}_{1}, \ldots, \mathbf{r}_{T})$. 

Given $\mathcal{F}_{t-1}$ denoting all previous historical data until $t-1$,  
$\mathbf{r}_t$ is a multivariate stochastic process of asset returns with conditional mean and covariance matrix denoted as 
\begin{align*} 
	\bm{\mu} = ~\mathbb{E}\left[\mathbf{r}_t | \mathcal{F}_{t-1}\right], \quad
	\mathbf{\Sigma} = \mathrm{Cov}\left[\mathbf{r}_t | \mathcal{F}_{t-1}\right]. 
\end{align*} 
At time $t$, the portfolio weight $\bbw_t = (w_{t}^{1}, \ldots, w_{t}^{d})$ is a function of the available information such as $\bbw_t = \bbw_t(\mathbf{r}_1, \ldots, \mathbf{r}_t)$.  
Thus, by continuously updating, an investor chooses the universally optimal portfolio, 
$\bbw^*= (w^{*,1}, \ldots, w^{*,d})$ for $t = 0,\ldots, T$ 
{given the following constraints: }
\begin{gather*}
	{\mathcal{S} :=\Big\{ \bbw = (w^1, \cdots, w^d)  \in \mathbb{R}^d~ \Big|~ \sum_{i=1}^{d}  w^i=1, \quad w^i\geq 0 \quad  i=1,\cdots, d\Big\}.}
\end{gather*}
The constraint explains the sum of portfolio weights should be one, and all portfolio weights are between 0 and 1. 

Given the evolutionary process of the wealth in \eqref{wealthprocess}, {an investor would choose his/her portfolio $ \bf w $. Now,} we introduce the objective function of an investor with portfolio weight $\bbw$. 
His objective is to maximize the expected utility of the terminal wealth $\mathcal{W}_T$ given the portfolio's variance so that 
\begin{align} \label{obj_value}
\underset{\bbw\in\mathcal S}{\mbox{max }} ~\mathbb{E} \left[U\left(\mathcal{W}_T\right)\right],
\end{align} 
where $\mathcal{S}$ is the feasible set of a portfolio weight vector $\bbw$ and $\mathbb{E}$ is the expectation operator. Here, we assume that investor's utility function $U(\cdot)$ is a Constant Relative Risk-Averse (CRRA) type utility function:
\[ U'(\cdot)>0 \quad \mbox{and} \quad U''(\cdot)<0. \]
{For example, one of the most commonly used CRRA-type utility functions is a log function.}
By the properties of the CRRA utility and statistical properties of securities given the investing time period, Markowitz \cite{M} has solved the optimal portfolio problem of \eqref{obj_value} by the mean-variance analysis, which maximizes the expected portfolio returns $ \bbw^\top\mu-r_f $ given the fixed volatility or risk of the portfolio at every $t$. {Given the maximum expected return relative to risk, a portfolio-investor would achieve the maximum level of wealth at the terminal period $ T $.}

 The ratio of this expected return relative to a risk
{{\[ L(\bbw):=\frac{\mathbf{w}^\top{\bm\mu}}{\sqrt{\mathbf{w}^\top\mathbf{\Sigma}\mathbf{w}}} \]}
is often called the Sharpe ratio,
where $\bbw^\top \bm{\mu}$ and $\sqrt{\bbw^\top\mathbf{\Sigma}\bbw}$ denote risky portfolio's expected return and  standard deviation of the portfolio $\bbw$ with multi-asset returns' covariance ${\bf\Sigma}$, respectively. This Sharpe ratio is the most commonly used objective function in the portfolio optimization literature.
 }

\subsection{Numerical simulations}
In this subsection, we provide numerical simulation results. For this, we choose six assets: Apple Inc.(APPL), Microsoft Inc.(MSFT), Starbucks Inc.(SBUX), Tesla Inc.(TSLA),  
German Deutsche Bank(DB), and gold(GLD) as a commodity via an exchange-traded fund(ETF). 
We use the daily closing prices over the period from Jan. 2019 to Nov. 2020. 
Figure \ref{fig:price} shows the stock prices and returns over the sample period. 

\begin{figure}[htbp]		
	\subfigure[Stock price evolution]{\label{fig:price1}
	\includegraphics[scale=0.8]{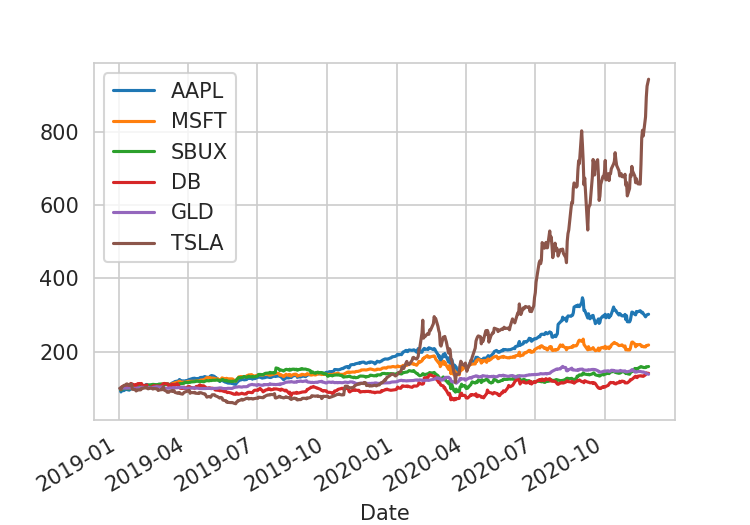}}
	\subfigure[Stock price returns]{\label{fig:return2}
	\includegraphics[scale=0.8]{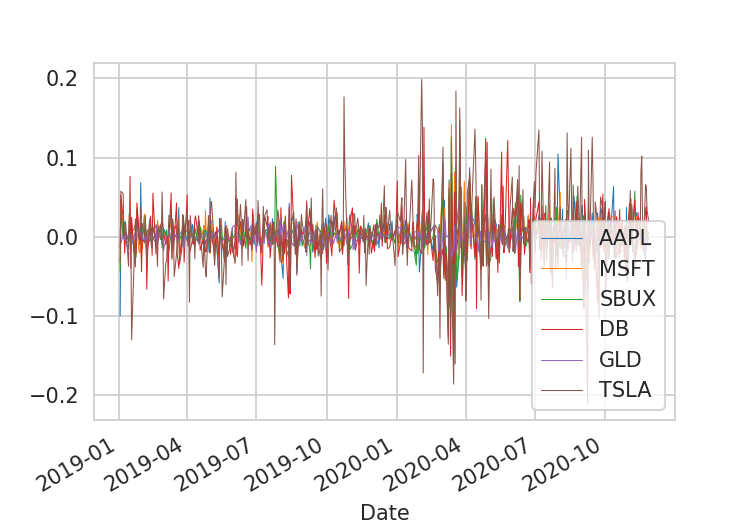}}
	\caption{Sample path of stock prices of APPLE, MSFT, SBUX, DB, GLD and TSLA.  Source: Yahoo Finance. 
	All initial prices are normalized to 100.} 
	\label{fig:price}	
\end{figure} 

From stock price data, we get daily returns of securities by taking the log differences of the stock prices. { The returns are continuously compounding and stationary} such as 
\begin{align*} 
r_{t}^i &= \log \frac{S_t^i}{S_{t-1}^i}.
\end{align*} 
The time index $t$ can denote any arbitrary period and we choose $\Delta t = 1$ day. 
In Figure \ref{fig:opt1} and \ref{fig:opt2}, we illustrate the results of the Monte Carlo simulation to generate 
random portfolio weight vectors on a large scale given $r_f = 0$. 
The initial configuration of $ {\bbw}_0 $ is sampled via Gaussian distribution and then projected to $ \mathcal{S}$ so that  $ \sum_{n} w_0^{i,n} =1  $ and $ w_0^{i,n}>0 $ for all $ i $. Parameter values are: $\Delta t = 0.01,\quad N=100, \quad \lambda=0.5, \quad \sigma=1$. 
{ In our experiment, we used Sharpe ratio as our objective function.}

For a fixed risk level, we may find multiple portfolios that all show different returns. 
In Figure \ref{fig:opt1} and \ref{fig:opt2}, scattered plots indicate alternative portfolio points of six assets 
over the risk-return space. 
Among admissible points in Figure \ref{fig:opt1}, 
the optimal risky asset portfolios are on the envelope because 
an investor aims to achieve the highest return given a fixed risk level (the minimum risk). 
The envelope is called the efficient frontier. 

\begin{figure}[htbp]
	\centering	
	\subfigure[Zoom-in Efficient Frontier by Sequential Least Squares]{\label{fig:opt1}
	\includegraphics[scale=0.7]{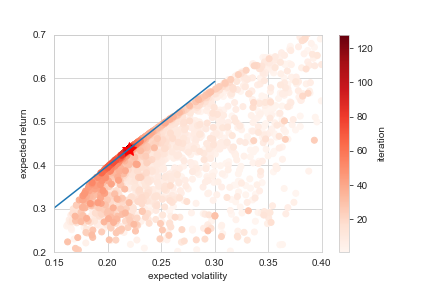}}
	\subfigure[Optimal Portfolio Selection by CBO Algorithm]{\label{fig:opt2}
	\includegraphics[scale=0.7]{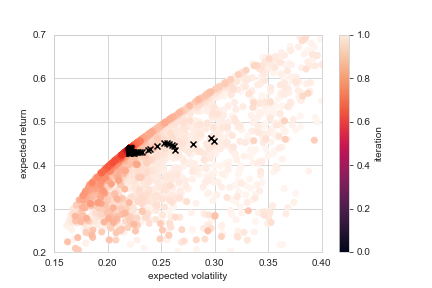}}
\caption{Optimal Portfolio Selection Simulation Results on the two-step CBO algorithm. The initial configuration of $ {\bbw}_0 $ is sampled by Gaussian distribution after projection to $ \mathcal{S}$ so that  $ \sum_{n} w_0^{i,n} =1  $ and $ w_0^{i,n}>0 $ for all $ i $. Parameter values are: $\Delta t = 0.01,\quad N=100, \quad \lambda=0.5, \quad \sigma=1$. }
\end{figure}

\begin{figure}[ht] 
	\centering  
	\includegraphics[scale=0.7]{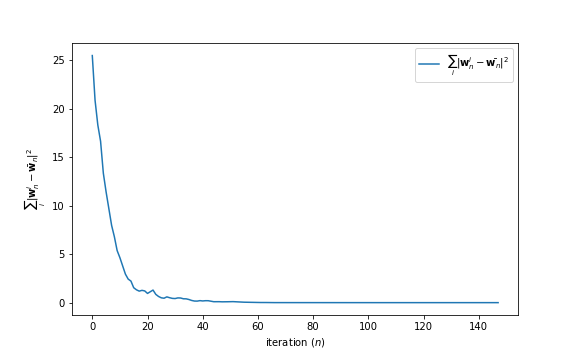}
	\caption{Time evolution of errors}\label{time_err}
\end{figure}

In addition to risky securities like stocks and commodities, one may add risk-free assets to his/her portfolio. 
Given the efficient frontier of risky asset bundle, an investor may choose his optimal weight on the risk-free asset relative to the risky asset bundle. 
How much weight to be invested in risk-free asset depends on his/her attitude toward risks. 
We describe the linear combination of the risk-free asset and risky bundle by the Capital Market Line (CML) in Figure \ref{fig:opt1}, 
a tangent line of the efficient frontier going through the risk-free rate of $r_f = 0.01$. 
Thus, any point on the CML indicates the portfolio that can achieve the maximum Sharpe ratio.   
If he is extremely risk-averse, then his optimal portfolio is found on $r_f$, whereas 
he could invest 100\% of his funds on the optimal risk-bundle lying on the efficient frontier, and such point is described by a red star in Figure \ref{fig:opt1}. 
A red star is found by the Sequential Least Squares Programming (SLSQP) optimizer using the Han-Powell quasi-Newton method with a BFGS update.

Finally, in Figure \ref{fig:opt2}, black crossed marks trace out the evolutionary path of the solutions' center of mass by the CBO algorithm \eqref{dCBO2}. The path shows that center of mass of $\{{\bbw}_t^{i}\}$ converge to ${\bbw}_\infty$ near ${\bbw}^* $ by iteration. We provide a zoomed-in picture on the global consensus with the red star obtained by the SLSQP optimizer. Figure \ref{time_err} shows the time evolution of $L_2$ error. As shown in Figure \ref{time_err}, the error decays over time.
Simulation results show that the predictor-corrector CBO algorithm successfully finds the optimal portfolio weight.}


\begin{thebibliography}{00}
	\bibitem{A-B-F} Albi, G., Bellomo, N., Fermo, L., Ha, S.-Y., Pareschi, L., Poyato, D. and Soler, J.: \textit{Vehicular traffic, crowds, and swarms. On the kinetic theory approach towards research perspectives.} Math. Models Methods Appl. Sci. {\bf 29} (2019), 1901-2005.
	
	\bibitem{Be} Bertsekas, D.: \textit{Convex Analysis and Optimization.} Athena Scientific. 2003.
	
	\bibitem{Bo} Bottou, L.: \textit{Online learning and stochastic approximations.} On-line Learning in
Neural Networks, {\bf 17} (1988). 142. 
	
	\bibitem{C-J-L-Z} Carrillo, J. A., Jin, S., Li, L. and Zhu, Y.: \textit{A consensus-based global optimization method for high dimensional machine learning problems.} ESAIM: COCV 27 (2021) S5.
	
	\bibitem{C-C-T-T} Carrillo, J., Choi, Y.-P., Totzeck, C. and Tse, O.: \textit{An analytical framework for consensus-based global optimization method.} Mathematical Models and Methods in Applied Sciences {\bf 28} (2018), 1037-1066.

\bibitem{C-H-L} Choi, Y.-P., Ha, S.-Y. and Li, Z.: \textit{Emergent dynamics of the Cucker-Smale flocking model and its variants.} In N. Bellomo, P. Degond, and E. Tadmor (Eds.), Active Particles Vol.I Advances in Theory, Models, and Applications, Series: Modeling and Simulation in Science and Technology, Birkhauser, Springer. 2017.

	\bibitem{E-K} Eberhart, R. and Kennedy, J.: \textit{Particle swarm optimization.} Proceedings of the IEEE International Conference on Neural Networks {\bf 4} (1995), 1942-1948.
	
	\bibitem{G-C-H-Z} Gong, C. Chen, H., He, W. and Zhang, Z.: \textit{Improved multi-objective clustering algorithm using particle swarm optimiation.} PLoS ONE. {\bf 12} (2017), 0188815.
	 
	\bibitem{G-P} Grassi, S. and Pareschi, L.: \textit{From particle swarm optimization to consensus based optimization: stochastic modeling and mean field limit.} Available at: https://arxiv.org/abs/2012.05613.
	
	\bibitem{H-J-K2} Ha, S.-Y., Jin, S. and Kim, D.: \textit{Convergence of a first-order consensus-based global optimization algorithm.} Mathematical Models and Methods in Applied Sciences {\bf 30} (2020), 2417-2444.
	
	\bibitem{H-J-K} Ha, S.-Y., Jin, S. and Kim, D.: \textit{Convergence and error estimates for time-discrete consensus-based optimization algorithms.} Numerische Mathematik {\bf 147} (2021), 255-282.
	
	
	\bibitem{Ho} Holland, J. H.: \textit{Genetic algorithms.} Scientific American {\bf 267}, 66-73 (1992).
	
	
	\bibitem{L-A} Laarhoven, P. J. M. van and Aarts, E. H. L.: \textit{Simulated annealing: theory and applications.} D. Reidel Publishing Co., Dordrecht, 1987.
	
	\bibitem{M} Markowitz, H.: \textit{Portfolio selection.} Journal of Finance {\bf 7} (1952), 77-91. 
	
	\bibitem{N} Nielsen, F.: \textit{Introduction to HPC with MPI for Data Science}. Springer. 2016.
		
	\bibitem{P-T-T-M} Pinnau, R., Totzeck, C., Tse, O. and Martin, S.: \textit{A consensus-based model for global optimization and its mean-field limit}. Math. Models Methods Appl. Sci. {\bf 27} (2017), 183-204.
	
	\bibitem{P-K-B} Poli, R., Kennedy, J. and Blackwell, T.: \textit{Particle swarm optimization: An overview}. SWARM INTELL-US. {\bf1} (2007) 33-57.  
	
	\bibitem{T-P-B-S} Totzeck, C., Pinnau, R., Blauth, S. and Schotth\'{o}fer, S.: \textit{A numerical comparison of consensus-based global optimization to other particle-based global optimization scheme.} Proceedings in Applied Mathematics and Mechanics, {\bf 18}, 2018.
		
	\bibitem{Ya} Yang, X.-S.: \textit{Nature-inspired metaheuristic algorithms.} Luniver Press, 2010.
	
	\bibitem{Y-D} Yang, X.-S., Deb, S., Zhao, Y.-X., Fong, S. and He, X.: \textit{Swarm intelligence: past, present and future}. Soft Comput {\bf 22} (2018), 5923-5933. 

	
\end{thebibliography}
\end{document}